\documentclass[12pt,a4paper,reqno]{amsart}

\usepackage{amssymb}
\usepackage{amsmath}
\usepackage[all]{xy}
\usepackage{latexsym}
\usepackage{amsthm}
\usepackage{a4wide}
\usepackage{color}
\usepackage{hyperref}
\input diagxy

\theoremstyle{definition}
\newtheorem{definition}{Definition}[section]
\newtheorem{remark}[definition]{Remark}
\theoremstyle{plain}
\newtheorem{lemma}[definition]{Lemma}
\newtheorem{theorem}[definition]{Theorem}
\newtheorem{example}[definition]{Example}
\newtheorem{corollary}[definition]{Corollary}
\newtheorem{proposition}[definition]{Proposition}
\newtheorem{question}[definition]{Question}

\begin{document}

\title{Categorically Proper Homomorphisms of Topological Groups}

\author{Wei He}
\address{Wei He\\Institute of Mathematics, Nanjing Normal University, Nanjing 210046, China}
\email{weihe@njnu.edu.cn}

\author{Walter Tholen}
\address{Walter Tholen\\Department of Mathematics and Statistics, York University, Toronto, Ontario, M3J 1P3, Canada}
\email{tholen@mathstat.yorku.ca}

\begin{abstract}
We extend the Dikranjan-Uspenskij notions of c-compact and h-complete topological group to the morphism level, study the stability properties of the newly defined types of maps, such as closure under direct products, and compare them with their counterparts in topology. We assume Hausdorffness only when our proofs require us to do so, which leads to new results and the affirmation of some facts that were known in a Hausdorff context.
\end{abstract}

\thanks{The first author acknowledges the support of NSFC (11171156, 11571175) and NSERC (501260). The second author acknowledges partial financial assistance by NSERC}
\date{}

\dedicatory{Dedicated to the memory of Horst Herrlich}

\maketitle

\section{Introduction}

Compactness of a topological space $X$ may be characterized by the property that the projection $X\times Z\to Z$ maps closed subspaces onto closed subspaces, for every space $Z$. Following the early categorical treatments of compactness in \cite{Herrlich, Penon, Manes, HSS} that are largely based on this essential property, the categorical development of closure operators (see \cite{DG, DT}) naturally led to a compelling notion of compact object in a category equipped with a closure operator (see \cite{Cas, DG2, CGT, CGT2}) or, more generally, with an axiomatically given class of ``closed morphisms" as in \cite{Tholen99, CGT2}. Pursuing this approach in the category of topological groups, Dikranjan and Uspenskij \cite{DU} called an object $G$ {\em c(ategorically) compact} if the projection $G\times K\to K$ is {\em c-closed, i.e.,} maps closed sub{\em groups} onto closed subgroups.
While (topologically) compact groups are trivially c-compact, and while various relevant additional properties (such as Abelianess) guarantee the validity of the converse statement (see \cite{DU, L}), its failure in general was shown only fairly recently in \cite{KOO}.

One of the main advantages of the categorical treatment of compactness in topology is the fact that it makes precise that
the notion of proper map (as promoted by \cite{Bourbaki}, often called perfect in general topology -- see \cite{En}) is
simply the fibred version of compactness, so that any categorically provable statement for compactness leads to a statement on proper
maps, and conversely: $f:X\to Y$ is proper if, and only if, $f$ as an object in the category of fibred spaces over $Y$, is compact; that is, if the projection $X\times_Y Z\to Z$ is closed, for every continuous map $g:Z \to Y$; equivalently, if $f\times 1_Z:X\times Z\to Y\times Z$ is closed for all spaces $Z$ (\cite{Bourbaki}), or if $f$ is closed and has compact fibres (\cite{En}). But do these equivalent formulations remain valid if we carry the categorical notion of properness
from ${\bf Top}$ to the category ${\bf TopGrp}$ of topological groups?

To answer these and further questions, in this paper we study the notion of {\em c(ategorical) properness} in the category {\bf TopGrp}, including its stability under arbitrary direct products, and compare it with the weaker notions
of {\em c-completeness} and {\em h(omomorphical) completeness} as introduced here at the morphism level as well, in generalization of the object notion of h-completeness studied in \cite{DU,L}. To derive product stability of c-proper and h-complete maps, we extend the characterizations of these notions in terms of convergence of special types of filters as given in \cite{DU, L} from the object to the morphism level. In addition,  we give an alternative proof for the product stability of c-proper maps based on the categorical Tychonoff Theorem \cite{CT} that had already been used
to affirm product stability of c-compactness for topological groups: see Example 9.5 of \cite{CGT}.

In fact, we not only extend but slightly generalize the known object-level results since, unlike the authors
of \cite {DU, L}
 and of most papers on topological groups, we do {\em not} assume {\em a priori} that the topologies of our objects must be Hausdorff. The reason for this, however, is not the aim for generalization {\em per se} but rather, we find that keeping track of where Hausdorffness plays a role adds to the clarity of proofs. In doing so, we follow the categorical literature which clearly
shows that compactness and separation properties are in many instances equal partners that mutually enhance each other. Of course, for maps Hausdorffness has to be understood in the fibred sense (see \cite{J}), and following
\cite{Tholen99, CGT2} we reserve the name {\em c-perfect} for c-proper maps that are also Hausdorff.

Section 2 presents the principal notion of this paper, c-proper map in {\bf TopGrp} and, based on the general categorical theory presented in \cite{Tholen99}, establishes its basic properties and interactions with c-compact objects and Hausdorffness. We highlight in particular the fact that, for maps
$f:G\to H$ and $ g:K\to H$ in {\bf TopGrp} with $f$ c-proper, c-compactness of $K$ implies c-compactness
of $G\times_H K$ (Corollary \ref{pullback ascent}); however, very unlike the situation in ${\bf Top}$,
a c-closed map $f$ satisfying this stability property for c-compactness for all $g$ may still fail to be c-proper (Example \ref{first example}).
Section 3 follows the lead of Section 2 in introducing h-complete and c-complete maps in {\bf TopGrp} and studying their stability properties. Employing an example given by Shelah \cite{SS} under CH, we show that, unlike for c-proper
maps, both types of maps may fail to be stable under composition (Corollary \ref{non-compositivity}).
We were, however, not able to determine whether the categorically preferable notion of c-completeness is truly stronger than that of h-completeness -- but conjecture that it is.

Having extended the existing filter characterizations for c-compact and h-complete objects as given in \cite{DU,L} to morphisms in Section 4, always keeping carefully track of, and trying to minimize, the assumption of Hausdorffness,
in Section 5 we reap the benefits and formulate the resulting product stability for c-proper and for h-complete maps,
thus establishing the {\bf TopGrp} counterparts for the classical Frol\'ik-Bourbaki \cite{Frolik, Bourbaki} and Chevalley-Frink \cite{CF} theorems in topology. However, our filter characterizations require that the codomains of the participating maps be Hausdorff. For product stability of c-proper maps, we were able to avoid this (however small) restriction, by applying the categorical results of \cite{CT, CGT} and thus demonstrating their power in the current context; see Corollary \ref{Frolik categorical}.

\section{Categorically closed, proper and perfect maps}

Throughout this paper we consider the category $\bf{TopGrp}$ of ({\em not necessarily Hausdorff}) topological groups and their continuous homomorphisms. Hence, {\em $G, H, K, ...$ will always denote topological groups,
and when we write $f:G \to H, g:H \to K, ...$, it is always assumed that $f, g, ...$ are continuous homomorphisms. In Section {\rm 4} we will also consider not necessarily continuous homomorphisms; they will always be denoted by Greek letters:
$\varphi: G \to H, \psi: H \to K, ...$, and when they are assumed to be continuous, we will say so explicitly. $A \le G$ means that $A$ is a subgroup of $G$ provided with the subspace topology. The identity map of $G$ is denoted by $1_G$, while $e_G$ denotes the neutral element of $G$.}

Whenever Hausdorffness is assumed, we will say so explicitly. Topological groups whose topology is Hausdorff are called {\em Hausdorff groups}. We remind the reader that $f: G \to H$ is {\em(fibrewise) Hausdorff} (see \cite{J}) if distinct points of the same fibre of $f$ may be separated by disjoint open neighbourhoods in $G$; equivalently, if the diagonal map $\delta_{f}: G \to G \times_H G$
is closed. Trivially, $G$ is Hausdorff if, and only if, $G \to 1$ is Hausdorff, with $1$ denoting a trivial group. Also, any $f:G\to H$ that is injective or has Hausdorff domain $G$ is Hausdorff as a map.

\begin{definition} \label{basic definition}
(1) $f: G \to H$ is {\em categorically closed} ({\em c-closed}, for short), if $f(A) \le H$ is closed for all closed $A \le G$; equivalently, if taking images of subgroups under $f$ commutes with taking closures: $f(\overline{A})=\overline{f(A)}$ for all $A \le G$.

(2) $f: G \to H$ is {\em categorically proper} ({\em c-proper}) if $f$ is stably c-closed, that is: if in every pullback diagram
\[
\xymatrix{{G \times_HK} \ar[r]^{p_{2}}\ar[d]_{p_{1}}& K
\ar[d]^{g}\\
      G  \ar[r]^{f}& H}
\]
the projection $p_2$ is c-closed.

(3) $f:G\to H$ is {\em categorically perfect} ({\em c-perfect}) if $f$ is c-proper and Hausdorff.

(4) (\cite{DU}) $G$ is {\em categorically compact} ({\em c-compact}) if $G \to 1$ is c-proper; equivalently: if the projection $p_2: G \times K \to K$ is c-closed, for every $K$.

\end{definition}

\begin{remark} \label{first remark}
(1) Obviously, every (topologically) closed morphism is c-closed, and every (topologically) proper (i.e., closed, with compact fibres) morphism is c-proper. In particular, every compact group is c-compact.

(2) The existence of a c-compact group $G$ that fails to be compact (see \cite{KOO}) shows that none of the implications of (1) is reversible. But we recall from Theorem 3.7 in \cite{DU} that an Abelian c-compact Hausdorff group must be compact. Example 2.1 below presents a c-closed morphism of Abelian Hausdorff groups that fails to be c-proper.

(3) When $f: G \to H$ is a topological embedding, so that $G\cong f(G) \le H$, then c-closedness of $f$ is equivalent to its closedness, i.e., to $f(G)\le H$ being closed.

(4) The notion of c-compactness as given in Definition \ref{basic definition}(4) is formally stronger than the one defined
in \cite{DU} and used in \cite{L}, even when the group $G$ is Hausdorff, since in our definition $K$ ranges over {\em all} topological groups, including those that are not Hausdorff. However, we will show in Corollary \ref{equivalence with DU} below that it suffices to consider Hausdorff groups $K$ in Definition \ref{basic definition}(4), so that the definition given here is in fact equivalent to the Dikranjan-Uspenskij notion when $G$ is Hausdorff.
\end{remark}

Let us first observe the following easy, but fundamental properties of the class $\mathcal F$ of c-closed morphisms:

(F0) {\em every isomorphism (in $\bf{TopGrp}$) is in $\mathcal F$;}

(F1) {\em if $f:G\to H, g:H \to K$ are in $\mathcal F$, so is $g \cdot f$;}

(F2) {\em for $g$ injective, if $g \cdot f$ is in $\mathcal F$, so is $f$;}

(F3) {\em for $f$ surjective, if $g \cdot f$ is in $\mathcal F$, so is $g$.}

For {\em any} class $\mathcal{F}$ of morphisms in $\bf TopGrp$, we say that $\mathcal{F}$ satisfies the {\em basic stability properties (BSP)} if (F0)--(F3) hold.
For example, the classes of open morphisms and of (topologically) closed morphisms both satisfy BSP. With the trivial observations that surjectivity is stable under pullback and that for $g$ injective
\[
\xymatrix{{M} \ar[r]^{1_{M}}\ar[d]_{h}& M
\ar[d]^{g\cdot h}\\
      G  \ar[r]^{g}& H}
\]
is a pullback diagram, pullback pasting easily gives that, for any $\mathcal{F}$ satisfying BSP, the class $\mathcal{F}^*$ of all those morphisms belonging stably to $\mathcal{F}$ also satisfies BSP. Since for $\mathcal{F}$ the class of c-closed morphisms $\mathcal{F}^*$ is the class of c-proper morphisms, this means in particular:

\begin{proposition} \label{c-proper BSP}
The class of c-proper morphisms satisfies the basic stability properties.
\end{proposition}

\begin{corollary} \label{necessary condition}
For $f:G\to H$ c-proper and any $K$ also $f\times 1_K:G\times K \rightarrow H\times K$ is c-proper.
\end{corollary}

\begin{proof}
$f\times 1_K$ is a pullback of $f$.
\end{proof}

Another useful consequence of BSP for c-proper maps is the following:

\begin{corollary} \label{induced arrow}
For  $f: G \rightarrow H$ c-proper
and any $g: G \rightarrow K$ with $K$ Hausdorff, the induced
morphism
  $\langle f, g\rangle: G \rightarrow H\times K$ is also c-proper.
  \end{corollary}

  \begin{proof}
  $\langle f, g\rangle$ factors as
  \[G \stackrel{\langle 1_{G}, g\rangle}
  {\longrightarrow} G \times K \stackrel{f \times 1_{K}}{\longrightarrow}
  H\times K\]
 with $f\times 1_K$ c-proper by Corollary 2.1 and $\langle 1_{G}, g\rangle$ a closed embedding since $K$ is Hausdorff.
 \end{proof}

We are also ready to confirm that c-propriety of $f:G\to H$ may be characterized in the following more familiar and easy form whenever $H$ is Hausdorff:

\begin{proposition} \label{traditional criterion}
Let $H$ be Hausdorff. Then $f:G\to H$ is c-proper if, and only if, $f\times 1_K: G\times K \to H\times K$ is c-closed, for every topological group $K$.
\end{proposition}

\begin{proof}
Any pullback diagram as in Definition \ref{basic definition}(2) may be decomposed as
\[\bfig
\square<500,400>[{G\times_H K}`K`{G\times K}`{H\times K};{p_2}`i`{\langle g,1_K\rangle}`{f\times 1_K}]
\square(0,-400)<500,400>[{G\times K}`{H\times K}`G`H;`{}`{}`f]
\efig
\]
where the embedding $i$ is closed since $H$ is Hausdorff, and where $f\times 1_K$ is c-closed by hypothesis. Since $\langle g,1_K\rangle$ is injective, with (F1), (F2) one concludes that $p_2$ is c-closed, as desired.
\end{proof}

For purely categorical reasons one has the same interplay between the object notion of c-compactness and the morphism notion of c-propriety as one has for the corresponding topological notions, i.e, when ``c-" is removed. We repeat the argumentation given in \cite{Tholen99} in the current environment:

\begin{theorem} \label{c-compact}
The following assertions for a topological group $G$ are equivalent:
\begin{itemize}
\item[\rm (i)] $G$ is c-compact;
\item[\rm (ii)] every $f:G \to H$ with $H$ Hausdorff is c-proper;
\item[\rm (iii)] there is a c-proper $f:G \to H$ with $H$ c-compact.
\end{itemize}
\end{theorem}

\begin{proof}
For (i)$\Rightarrow$(ii) factor $f$ through its graph as  in
\[
\xymatrix{
 G \ar[rr]^{<1_G,f>}\ar[dr]_<<<<<{f}& & {G \times H} \ar[dl]^<<<<<{p_2}\\
 & H  &}
\]
and apply (F1). For (ii)$\Rightarrow$(iii) observe that $H=1$ is c-compact Hausdorff, and for (iii)$\Rightarrow$(i) apply (F1) to the trivial diagram
\[
\xymatrix{
 G \ar[rr]^{f}\ar[dr]_<<<<<{}& & {H} \ar[dl]^<<<<<{}\\
 & 1  &}
\] \end{proof}

\begin{remark} \label {categorical proof}
(1) As the proof above relies only on the fact that the relevant class $\mathcal{F}$ satisfies (F0) and (F1), the argumentation carries through, for example, also for (topologically) closed morphisms and for Hausdorff morphisms.

(2) More importantly, we may apply (the categorical version of) Theorem \ref{c-compact} to the category $\bf{TopGrp}$$/K$ of topological groups over the fixed topological group $K$, noting that the class of c-closed morphisms over $K$ still satisfies BSP. One may then regard $h:G \to K$ in $\bf{TopGrp}$ also as the unique morphism in $\bf{TopGrp}$$/K$ from $(G,h)$ to the terminal object $(K,1_K)$ in $\bf{TopGrp}$$/K$. Hence, $h$ is c-proper in $\bf{TopGrp}$ if, and only if, $(G,h)$ is c-compact in $\bf{TopGrp}$$/K$. Consequently, as previously stated in the abstract context in \cite{Tholen99}, we obtain the fibred version of Theorem 2.1 (i)$\Rightarrow$(ii) as given in (1) of the next Corollary, which strengthens considerably the assertion (F2) for c-proper maps, as follows:

\end{remark}

\begin{corollary} \label{Hausdorff cancellation}
{\rm (1)} If for $f:G \to H$ and $g:H \to K$ the composite $g\cdot f$ is c-proper and $g$ is Hausdorff, then  $f$ is c-proper.

{\rm (2)} Let $A\le G$ be dense, with $G$ Hausdorff. Then a c-proper morphism $h:A \to H$ cannot be extended to a morphism $G\to H$, unless $A=G$.
\end{corollary}

\begin{proof}
For (1) see Remark \ref{categorical proof}(2), and (2) follows from (1).
\end{proof}

\begin{remark} \label{surjective cancellation}
Since c-proper morphisms satisfy (F3), so that $g \cdot f$ c-proper with $f$ surjective implies that also $g$ is c-proper, one has in particular that for every surjective morphism $f:G\to H$ with $G$ c-compact also $H$ is c-compact. More remarkably, as previously stated (see \cite{Tholen99}) in generalization of the corresponding topological fact (without the ``c-"), Hausdorffness is transferred along surjective c-proper morphisms, as follows.
\end{remark}

\begin{proposition} \label{Hausdorff transfer}
Let $f:G\to H$ be c-proper and surjective. Then, if $G$ is Hausdorff, so is $H$. More generally, any $g:H\to K$ with $g \cdot f$ Hausdorff must be Hausdorff itself.

\end{proposition}

\begin{proof}
For the first assertion, consider the commutative diagram
\[
\xymatrix{{G} \ar[r]^{f}\ar[d]_{\delta_G}& H
\ar[d]^{\delta_H}\\
      {G\times G}  \ar[r]^{f \times f}& {H\times H}}
\]
and, noting $f \times f=(f\times 1_H)\cdot (1_G\times f)$, apply (F1), (F3). The second statement is the fibred version of the first.
\end{proof}

\begin{corollary}\label{perfect BSP}
The class of c-perfect maps satisfies the basic stability properties.
\end{corollary}

The following Corollary is another easy, but significant, consequence of Theorem \ref{c-compact} obtained by categorical reasoning:

\begin{corollary} \label{pullback ascent}
In any pullback diagram
\[
\xymatrix{{G \times_HK} \ar[r]^{p_2}\ar[d]_{p_1}& K
\ar[d]^{g}\\
      G  \ar[r]^{f}& H}
\]
with $f$ c-proper, if $K$ is c-compact, so is $G\times_H K$.

\end{corollary}

\begin{proof}
Apply Theorem \ref{c-compact} (iii)$\Rightarrow$(i) to $p_2$ in lieu of $f$.
\end{proof}

As a consequence, the inverse image of every c-compact subgroup of $H$ under the c-proper map $f:G\to H$ is c-compact; in particular, its kernel is c-compact. But neither of these necessary conditions for c-propriety of $f$ is also sufficient; not even the property described by Corollary \ref{pullback ascent} is strong enough to imply c-propriety of $f$:

 \begin{example}\label{first example}
 There is a c-closed morphism $f:G\to H$ of Abelian Hausdorff groups that fails to be c-proper but still has the property that for every $g:K\to H$ with $K$ c-compact also $G \times_HK$ is c-compact.
 \end{example}

 \begin{proof}

For a fixed prime number $p,  p> 3$, let $G$ be the additive subgroup $\{{{k}\over {p^{n}}} \mid k, n \in \mathbb{Z}\}$ of the Euclidean line $\mathbb{R}$. The key property of $G$ to be used is that every subgroup $A\le G$ different from $G$ is closed. Hence, with
$G_{d}$ denoting the same underlying group as $G$ provided with the discrete topology, if we let $f: G_{d} \to G$ map identically, then $f$ is trivially c-closed but, as we show next, not c-proper.
Indeed, consider $f \times 1_{\mathbb{R}}: G_{d} \times \mathbb{R} \rightarrow G \times \mathbb{R}$ and fix a sequence of rational numbers $(a_{n})$ converging to an irrational number $x$.  Let $H = \overline{\langle ({{1}\over {p^{n}}},  a_{n})\rangle}$ be the closure of the subgroup of $G_{d} \times \mathbb{R}$ generated by
$\{({{1}\over {p^{n}}}, a_{n}) \mid n \in \mathbb{N}\}$. $H$ is not closed  in $G \times \mathbb{R}$ since $(0, x) \not\in H$.

Let us now consider any $g: K \rightarrow G$ with $K$ c-compact. Then the image $g(K)\le G$ is also c-compact and, in fact, compact, since it is Abelian. But $G$ has no non-trivial compact subgroup, so $g$ must be constant. Consequently, the pullback of $g$ along $f$ is the constant morphism $K \rightarrow G_{d}$, and we can conclude that $G_{d}\times_GK\cong K$ is c-compact.
\end{proof}

\begin{remark}
Considering the quotient group $H=G/{\mathbb Z}$ with $G$ as in Example \ref{first example} we note that all proper subgroups of $H$ are finite. Denote by
$H_{d}$ the corresponding discrete group; then $H_{d}\to H$ which maps elements identically is still c-closed, but not only fails to be c-proper but also fails to be a (topologically) closed map.
\end{remark}

\section{Categorically and homomorphically complete objects and morphisms}

Recall that a Hausdorff group $G$ is complete (in its uniformity) if, and only if, it is absolutely closed, in the sense that it is closed in every Hausdorff extension $K\ge G$ (see \cite{AT}). Since every discrete topological group is complete, trivially completeness fails to be closed under taking homomorphic images. It therefore makes sense to pay attention to those topological groups whose homomorphic images are complete whenever they are Hausdorff. We introduce here the resulting notion of h-completeness more generally at the morphism level and contrast it with the new and formally stronger notion of c-completeness that seems to be categorically smoother, as follows:

\begin{definition} \label{definition c-complete}
(1) $f:G\to H$ is {\em homomorphically complete} ({\em h-complete}, for short) if in every factorization
\[
\xymatrix{
 G \ar[rr]^{h}\ar[dr]_<<<<<{f}& & {K} \ar[dl]^<<<<<{k}\\
 & H  &}
\]
with k Hausdorff, the image $h(G)\le K$ is closed.

(2) $f$ is {\em categorically complete} ({\em c-complete}) if in every factorization $f=k\cdot h$ with $k$ Hausdorff, $h$ must be c-closed.

(3) $G$ is {\em h-complete} (\cite{DT}) or {\em c-complete} if $G \to 1$ has the respective property.
\end{definition}

\begin{remark}\label{second remark}
(1) By defintion, $G$ is h-complete if, and only if, for every $h:G\to K$ with $K$ Hausdorff, $h(G)\le K$ is closed; equivalently, as indicated above, if for every {\em surjective} $h:G\to K$ with $K$ Hausdorff, $K$ is even complete. In particular, an h-complete Hausdorff group is complete.

(2) By definition, $G$ is c-complete if every $h:G\to K$ with $K$ Hausdorff is c-closed. One sees immediately that $G$ is c-complete if every closed subgroup of $G$ is h-complete. But, as we'll see in Example \ref{second example} below, the condition may not be necessary for c-completeness of $G$.

(3) Trivially, every c-complete morphism is h-complete; in particular, every c-complete object is h-complete.
A c-complete morphism is also c-closed (just consider $h=f, k=1_H$ in the defining property for c-completeness), and the image of an h-complete morphism is closed in the codomain (for the same reason).

(4) In general, h- or c-completeness of $f:G\to H$ does not imply the respective property for $G$. Indeed,
$1_G: G\to G$ is easily seen to be c- and, hence, h-complete, for any Hausdorff group $G$.

(5) By Corollary \ref{Hausdorff cancellation} every c-proper morphism is c-complete; in particular, every c-compact group is c-complete. For a refined statement, see Remark \ref{implications} below.

(6) It is easy to see that properties (F0), (F2), (F3) (see Section 2) hold for the classes of c-complete and of h-complete morphisms. However, as we'll show in Corollary \ref{non-compositivity} below, (F1) fails for both of these newly defined classes, so that neither of them satisfies BSP. Nevertheless, as we show first, the corresponding object and morphism notions interact similarly to the properties stated in Theorem \ref{c-compact} about c-compactness and c-propriety.
\end{remark}

\begin{theorem}\label{c-complete}
The following assertions for a topological group $G$ are equivalent:
\begin{itemize}
\item[\rm (i)] $G$ is c-complete;
\item[\rm (ii)] every $f:G \to H$ with $H$ Hausdorff is c-complete;
\item[\rm (iii)] there is a c-complete $f:G \to H$ with $H$ c-compact.
\end{itemize}

These equivalences remain valid if ``c-complete" is traded for ``h-complete" everywhere.
\end{theorem}

\begin{proof} (i)$\Rightarrow$(ii): Considering a factorization as in Definition \ref{definition c-complete}(1) with $k$ Hausdorff, we just observe that $K$ is Hausdorff when $H$ is Hausdorff.

(ii)$\Rightarrow$(iii): The trivial group 1 is c-compact Hausdorff.

(iii)$\Rightarrow$(i): Given $h:G\to K$ with $K$ Hausdorff one considers the factorization
\[
\xymatrix{
 G \ar[rr]^{\langle f,h\rangle}\ar[dr]_<<<<<{f}& & {H\times K} \ar[dl]^<<<<<{p_1}\\
 & H  &}
\]
in which $p_1$ is Hausdorff. Consequently, $\langle f,h\rangle$ is c-closed. Since $H$ is c-compact,
$p_2:{H\times K}\to K$ is c-closed, whence $h=p_2 \cdot \langle f,h\rangle$ is also c-closed.

For the ``h-closed" case the proof proceeds similarly.
\end{proof}

Having used exclusively categorical arguments in the proof above we may immediately conclude the fibred version of the implications (iii)$\Rightarrow$(i)$\Rightarrow$(ii) of Theorem \ref{c-complete}:

\begin{corollary}\label{compose c-complete}
For $f:G\to H$ and $g:H\to K$ one has:
\begin {itemize}

\item[\rm{(1)}] If $f$ is c-complete and $g$ is c-proper, then $g\cdot f$ is c-complete.

\item[\rm{(2)}] If $g\cdot f$ is c-complete and $g$ Hausdorff, then $f$ is c-complete.
\end{itemize}

These statements remain valid if ``c-complete" is traded for ``h-complete" everywhere.
\end{corollary}

\begin{remark}\label{special linear group}

(1) It has been noted in \cite{L} (Example 4.5) that a closed subgroup of an h-complete Hausdorff group may fail to be h-complete. Indeed, it is known (but certainlyly non-trivial) that the locally compact and minimal groups
${\rm SL}_n({\mathbb R})\,(n\ge 2)$ are h-complete.
(Recall that a Hausdorff group is {\em minimal} if it does not admit a strictly
coarser Hausdorff group topology.) The special linear group ${\rm SL}_2({\mathbb R})$ contains an isomorphic copy
of the discrete group $\mathbb Z$ as a closed subgroup. But $\mathbb Z$ is not h-complete since it is not complete in the $p$-adic topology.

(2) Considering ${\mathbb Z}\to {\rm SL}_2({\mathbb R})\to 1$ as in (1) we see that, in general, propriety of $f:G\to H$ and h-completeness of $g:H\to K$ do not guarantee that $g\cdot f$ be h-complete, even when $G$ is Abelian.
\end{remark}

Next we will show that not even c-propriety of $f$ and c-completeness of $g$ guarantee that $g\cdot f$ be h-complete, as long as we assume the {\em Continuum Hypothesis (CH)}. In particular, the class of c-complete morphisms (as well as that of h-complete morphisms) fails to be closed under composition. Some preparations are needed.

\begin{remark} \label{complete to c-complete}
(1) It is well known (see \cite{U}) that every Hausdorff group may be embedded into a complete, minimal and topologically simple Hausdorff group. (Recall that a Hausdorff group is {\em topologically simple} if it does not admit any non-trivial closed normal subgroups.) As a consequence, every complete Hausdorff group is a subgroup of a minimal and topologicaly simple complete Hausdorff group.

(2) One knows from \cite{DU} that every minimal and topologically simple complete Hausdorff group is h-complete.

(3) Every Abelian h-complete Hausdorff group is compact (see \cite{DU}).

\end{remark}

Assertion (2) may be strengthened, as follows.

\begin{proposition}\label{minimal}
Every minimal and topologically simple complete Hausdorff group is c-complete.
\end{proposition}

\begin{proof}
 Since our minimal and topologically simple complete Hausdorff group $G$ is h-complete (Remark \ref{complete to c-complete}(2)), it suffices to consider a surjective $h: G \rightarrow K$ with $K$ Hausdorff and show that $h$ must be c-closed. Simplicity of $G$ forces $h$ to be constant or injective; in the latter case $h$ is  bijective and, in fact, an isomorphism since $G$ is minimal. In both cases $h$ is (c-)closed.
\end{proof}

\begin{example}\label{second example}
Under CH there is a c-complete Hausdorff group $M$ with a closed Abelian subgroup $A$ that fails to be h-complete.
\end{example}

\begin{proof}
Assuming CH Shelah constructed in \cite{SS} a non-Abelian but torsion-free and simple group $M$ of
cardinality $\aleph_{1}$ that admits only the trivial group topologies. Thus $M$, endowed with
the discrete topology, is a minimal and (topologically) simple discrete group that is trivially complete and, consequently, by Proposition \ref{minimal}, actually c-complete.
Now consider any cyclic subgroup $A\le M$. Since $M$ is torsion-free, $A$ is infinite but distinct from the uncountable group $M$. Since $M$ is discrete, $A$ is closed in $M$ but cannot be compact and, in fact, can not even be h-complete, since it is Abelian: see Remark \ref{complete to c-complete}(3).
\end{proof}

\begin{corollary}\label{non-compositivity}
Under CH, the composite of a proper map followed by a c-complete map may fail to be h-complete.
\end{corollary}

\begin{proof}
Consider $A\to M \to 1$ of Example \ref{second example}.
\end{proof}

\begin{remark}\label{implications}
(1) Let us summarize the connections between some of the properties discussed for $f:G\to H$ so far, as follows. {\em Each of the following properties implies the next:
\begin{itemize}
\item[\rm (i)] $f$ is (topologically) proper;
\item[\rm (ii)] $f$ is c-proper;
\item[\rm (iii)] for every closed subgroup $S\leq G$, $f|_S$ is h-complete;
\item[\rm (iv)] $f$ is c-complete;
\item[\rm (v)] $f$ is h-complete.
\end{itemize}}
Indeed, for (i)$\Rightarrow$(ii) and (ii)$\Rightarrow$(iii), see Remark \ref{first remark}(1) and Corollary \ref{Hausdorff cancellation}(1) (to be applied to $f|_S$ in lieu of $f$), respectively, and (iii)$\Rightarrow$(iv)$\Rightarrow$(v) is trivial.

(2) That the implications (i)$\Rightarrow$(ii) and (iii)$\Rightarrow$(iv) are not reversible (at least under CH in the latter case) was already mentioned in Remarks \ref{first remark}(2) and \ref{second remark}(2), with reference to \cite{KOO} and Example \ref{second example}, respectively.

(3) Since c-completeness implies c-closedness (see Remark \ref{second remark}(3), a {\em stably} c-complete morphism (i.e., a map such that every pullback of it is c-complete) is necessarily c-proper.
\end{remark}

The Remark leaves us with the following questions.

\begin{question}\label{question}
\begin{itemize}
\item[\rm{(1)}] In Corollary {\rm{\ref{non-compositivity}}}, is the hypothesis CH essential?
\item[\rm{(2)}] Does h-completeness imply c-completeness? Specifically, is there an h-complete Hausdorff group that is not c-complete?
\item[\rm{(3)}] Is there an $f$ satisfying property \rm{(iii)} \em{of Remark \rm{\ref{implications}} \em{but failing to be c-proper?}} {\rm{(See also Remark \ref{discrete c-compact}(2) below.)}}
\item[\rm{(4)}] Is a {\rm stably} h-complete morphism necessarily c-proper?
\end{itemize}
\end{question}

\section{Characterizations with special filters}

In the context of Hausdorff groups, both c-compactness and h-completeness of $G$ have been characterized in terms of the existence of cluster points and of the convergence of special filters on $G$: see \cite{DU, L}. In this section we extend these characterizations to the morphism levels, continuing to keep Hausdorff assumptions at a minimum. Filters are always assumed to be proper. $\mathcal{N}(x)$ denotes the set of open neighbourhoods of $x\in G$.

\begin{definition}\label{definition g-filter}
Let $\mathcal{F}$ be a filter on $G$.

(1) $\mathcal{F}$ is a {\em group filter} ({\em g-filter}, for short) for $f: G\to H$ if there exist $r:P\to H$, $\varphi: G\to P$ and $x \in P$ such that $f = r\cdot\varphi$ and the family $\{\varphi^{-1}(U) \mid
 U \in \mathcal{N}(x)\}$ is a base of  $\mathcal{F}$.  (Note that, according to our notational conventions, $\varphi$ is just a homomorphism that, unlike $f$ and $r$, is {\em not} assumed to be continuous {\em a prori}!)

 (2) $\mathcal{F}$ is an {\em open-group filter} ({\em og-filter}, for short) for $f$ if $r, \varphi, x$ exist as in (1), but with $\varphi$ now being assumed to be also continuous.

 (3) $\mathcal{F}$ is an {\em (o)g-filter} on $G$ if $\mathcal{F}$ is an (o)g-filter for $G \to 1$.
 \end{definition}

\begin{remark}\label{remark g-filter}
(1) Every (o)g-filter for $f:G\to H$ is an (o)g-filter on $G$.

(2) In Definition \ref{definition g-filter}(1), one may always assume that the image of $\varphi$ is dense in $P$. Indeed, if one had $x \in U:= P\setminus \overline{\varphi(G)}$, then $\varphi^{-1}(U)=\emptyset$, so that $\mathcal{F}$ would not be proper.

(3) Every (o)g-filter on $G$ is contained in a maximal (w.r.t. $\subseteq$) (o)g-filter on $G$: see \cite{DU, L}. This may be used to establish the corresponding fact at the morphism level, as we show next.

\end{remark}

\begin{lemma}\label{lemma maximal}
For an (o)g-filter $\mathcal{F}$ for $f:G\to H$ and a compatible (o)g-filter $\mathcal{G}$ on $G$ (so that
$A \cap B \neq \emptyset$ for all $A\in \mathcal{F}, B\in \mathcal{G}$), the filter $\mathcal H$ generated by $\mathcal{F} \cup \mathcal{G}$ is an (o)g-filter for $f$.
\end{lemma}

\begin{proof}
We have $r, \varphi, x$ as in Definition \ref{definition g-filter}(1) and $\psi: G\to Q, y \in Q$ such that
$\{\psi^{-1}(V) \mid V \in \mathcal{N}(y)\}$ is a base of  $\mathcal{G}$. The diagram
\[
\xymatrix{
 G \ar[rr]^{\langle \varphi,\psi \rangle}\ar[dr]_<<<<<{f}& & {P\times Q} \ar[dl]^<<<<<{r\cdot p_1}\\
 & H  &}
\]
 commutes, and the filter $\mathcal H$ is actually generated by
 $\{\langle \varphi,\psi \rangle^{-1}(W) \mid W \in \mathcal{N}((x,y))\}$.
   \end{proof}

   \begin{corollary}\label{corollary maximal}
   Every (o)g-filter for $f:G\to H$ is contained in a maximal (o)g-filter for $f$.
   \end{corollary}

\begin{proof}
An (o)g-filter $\mathcal F$ for $f$ is an (o)g-filter on $G$ and therefore contained in a maximal (o)g-filter $\mathcal{G}$ on $G$ (see Remark \ref{remark g-filter}). By Lemma \ref{lemma maximal} $\mathcal G$ is actually an (o)g-filter for $f$ and certainly maximal also as such.
\end{proof}

\begin{proposition}\label{g-filter transfer}
Consider a g-filter $\mathcal F$ for $f$ and a commutative diagram
\[\bfig
\square<500,400>[G`M`H`N;{\pi}`f`{g}`{h}]
\efig
\]
in which the (not necessarily continuous) homomorphism $\pi$ is surjective and $N$ is Hausdorff.
Then the image $\pi(\mathcal{F})$ (generated by $\{\pi (A) \mid A \in \mathcal F \}$)  is a g-filter for $g$. Moreover,  if
$\mathcal F$ is an og-filter for $f$ and $\pi$ a topological quotient map, then $\pi(\mathcal{F})$ is an og-filter for $g$. Maximality transfers from $\mathcal F$ to $\mathcal G$ in both the g-filter and the og-filter case.
\end{proposition}

\begin{proof}
We have $r, \varphi, x$ as in Definition \ref{definition g-filter}(1) and may assume $P=\overline{\varphi(G)}$.
For $K$ the kernel of $\pi$, its image $\varphi(K)$ under $\varphi$ is normal in $\varphi(G)$, and then $L:= \overline{\varphi(K)}$ is normal in $P$ (see Cor. 1.19 of \cite{L}). So, we can form the group $Q:=P/L$ and provide it with the quotient topology, making the projection $p:P\to Q$ open. Since $p(\varphi(K))=\{e_Q\}$ there is $\psi:M\to Q$ (not necessarily continuous) with
$\psi \cdot \pi = p\cdot \varphi$, and since
\[ h(r(\overline{\varphi (K)})) \subseteq \overline{h(r(\varphi(K)))}= \overline{h(f(K))}= \overline{g(\pi(K))}=\overline{\{e_N\}}=\{e_N\} \]
(as $N$ is Hausdorff), there is $s:Q\to N$ with $s\cdot p=h\cdot r$ and, consequently, $s\cdot \psi =g$.
  \[\bfig
\square<1200,800>[G`M`H`N;\pi`f`g`h]
\Dtriangle<300,400>[G`P`H;f`\varphi`r]
\Ctriangle(900,0)<300,400>[M`Q`N;\psi``s]
\morphism(300,400)<600,0>[P`Q;p]
\efig\]

We now show that $\pi(\mathcal{F})$ generated by  $\{\pi(\varphi^{-1}(U))\mid U \in \mathcal{N}(x) \}$ coincides with the g-filter $\mathcal G$ for $g$ generated by the base $\{\psi^{-1}(V)\mid V \in \mathcal{N}(p(x)) \}$. One certainly has $\mathcal{G}\subseteq \pi(\mathcal{F})$ since
$\psi^{-1}(V)=\pi(\pi^{-1}(\psi^{-1}(V)))=\pi(\varphi^{-1}(p^{-1}(V)))$
for all $V\in\mathcal{N}(p(x))$. For the converse inclusion, let $U\in\mathcal{N}(x)$. Since $U\overline{A}=UA$ for any subset $A\subseteq P$, one has in particular $p^{-1}(p(U))=UL=U\varphi(K)$ and then
\[
\psi^{-1}(p(U))=\pi(\pi^{-1}(\psi^{-1}(p(U))))=\pi(\varphi^{-1}(p^{-1}(p(U))))=\pi(\varphi^{-1}(U\varphi(K)))=\pi(\varphi^{-1}(U))
\]
with $p(U)\in \mathcal{N}(p(x))$.

Clearly, when $\pi$ is a topologial quotient map and $\varphi$ is continuous, so is $\psi$. Hence, when $\mathcal{F}$ is an og-filter for $f$, $\mathcal{G}$ is one for $g$. Also, if $\mathcal F$ is maximal, whether it be as a g-filter or an og-filter for $f$, and if $\mathcal H$ is an (o)g-filter for $g$ with $\mathcal{G}=\pi(\mathcal{F})\subseteq \mathcal{H}$,
then $\mathcal{F}\subseteq \pi^{-1}(\mathcal{H})$, which implies $\mathcal{G}=\mathcal{H}$ by maximality of $\mathcal{F}$.
   \end{proof}

   \begin{theorem} \label{c-proper filter char}
The following assertions for $f:G\to H$ with $H$ Hausdorff are equivalent:
\begin{itemize}
\item[\rm (i)] $f$ is c-proper;
\item[\rm (ii)] for all $S\leq G$, every g-filter for $f|_S$ has a cluster point in $G$;
\item[\rm (iii)] for all $S\leq G$, every maximal g-filter for $f|_S$ converges in $G$.
\end{itemize}
\end{theorem}

\begin{proof}
(i)$\Rightarrow$(ii): For $S\leq G$ and a g-filter $\mathcal F$ for $f|_S$ one has $\varphi: S\to P,  r:P\to H$ and $x\in P$ with $r\cdot \varphi = f\cdot j$, and $\mathcal F$ generated by $\{\varphi^{-1}(U) \mid U\in \mathcal{N}(x) \}$; here $j$ is the inclusion map of $S$ into $G$, and $\varphi(S)$ may be assumed to be dense in $P$. Let
$T= \{(z,\varphi(z)) \mid z\in S\}$ be the image of $S$ in $G\times P$ under $\langle j,\varphi \rangle$. Continuity of
$\langle r,1_P \rangle: P \to H\times P$ and c-closedness of $f\times 1_P$ then give

$\langle r,1_P\rangle (P) \subseteq \overline{\langle r,1_P \rangle (\varphi (S))}=\overline{\langle f\cdot j,\varphi \rangle(S)}=\overline{(f\times 1_P)\cdot \langle j,\varphi \rangle (S)}=(f\times 1_P)(\overline{T})$.
\[\bfig
\square/<-`->`->`<-/<600,500>[G`G\times P`H`H\times P;p_1`f`f\times 1_P`q_1]
\square(600,0)<600,500>[G\times P`P`H\times P`P;p_2``1_P`q_2]
\Atrianglepair(0,500)/_(->`->`->`<-`->/<600,400>[S`G`G\times P`P;j`\langle j,\varphi\rangle`\varphi``]
\efig\]
Therefore, $(r(x),x)\in (f\times 1_P)(\overline{T})$, and we obtain $a\in G$ with $(a,x)\in \overline{T}$ and $f(a)=r(x)$. We claim that $a$ is a cluster point of $\mathcal F$, i.e., that $a\in \overline{\varphi^{-1}(U)}$ for
every $U\in \mathcal{N}(x)$. Indeed, for every $W\in \mathcal{N}(a)$ we have $(a,x)\in (W \times U)\cap T$ and,
consequently, $W\cap \varphi^{-1}(U)\neq \emptyset$.

(ii)$\Leftrightarrow$(iii) follows easily with Lemma \ref{lemma maximal} and Corollary \ref{corollary maximal}.

(ii)$\Rightarrow$(i): For a pullback square as in Definition \ref{basic definition}(2) we consider $S \leq G\times_H K$
closed and must show that $p_2(S)\leq K$ is also closed. Hence, we let $x\in \overline{p_2(S)}$ and form the filter
$\mathcal{F}$ on $S$ generated by $\{p_2^{-1}(U)\cap S \mid U\in \mathcal{N}(x)\}$. Clearly, $\mathcal F$ is a g-filter for $p_2|_S$, and with Proposition \ref{g-filter transfer} applied to
\[\bfig
\square<500,400>[S`p_1(S)`K`H;{p_1|S}`{p_2|_S}`{f|_{p_1(S)}}`{g}]
\efig
\]
we obtain that $p_1(\mathcal{F})$ is a g-filter for $f|_{p_1(S)}$. (Note that, while $\mathcal F$ is actually an og-filter, $p_1(\mathcal{F})$ may not be one.) By hypothesis then, $p_1(\mathcal{F})$ has a cluster point $a$ in $G$.
Since $S$ is closed, it now suffices to show $(a,x)\in \overline{S}$. Indeed, one obtains $(W\times U)\cap S\neq \emptyset$ for all $W\in \mathcal{N}(a), U\in \mathcal{N}(x)$, since $W\cap p_1(p_2^{-1}(U)\cap S)\neq \emptyset$, by virtue of $a$ being a cluster point of $p_1(\mathcal{F})$.
\end{proof}

The Theorem returns the known result at the object level (see \cite{DU, L}), but without the assumption of Hausdorffness:

  \begin{corollary} \label{c-compact filter char}
The following assertions for a topological group $G$ are equivalent:
\begin{itemize}
\item[\rm (i)] $G$ is c-compact;
\item[\rm (ii)] for all $S\leq G$, every g-filter on $S$ has a cluster point in $G$;
\item[\rm (iii)] for all $S\leq G$, every maximal g-filter on $S$ converges in $G$.
\end{itemize}
\end{corollary}

The corresponding characterization for h-complete morphisms is obtained quite easily once one has taken note of the
following Lemma utilizing the {\em Hausdorff reflection} of a topological group $G$, facilitated by the projection $\pi_G: G\to G/\overline{\{e_G\}}$.

\begin{lemma} \label{hausdorff lemma}
For $f: G\to H$ with $H$ Hausdorff and a g-filter $\mathcal{F}$ for $f$, there exist a Hausdorff group $Q$, $\psi: G\to Q$ and $k: Q\to H$ such that $f= k\cdot \psi$ and  $\mathcal{F}$ is generated by $\{\psi^{-1}(W) \mid W \in \mathcal{N}(y)\}$ for some $y \in Q$. When $\mathcal F$ is an og-filter for $f$, then $\psi$ may be chosen to be continuous.
\end{lemma}

\begin{proof} The g-filter $\mathcal F$ for $f$ comes with $r, \varphi, x$ as in Definition \ref{definition g-filter}(1). Let $Q = P/\overline{\{e_P\}}$ be the Hausdorff reflection of $P$ with projection $\pi_P: P\to Q$.  Since $H$ is Hausdorff, $r:P\to H$ factors through $Q$ as $r=k\cdot\pi_P$, so that with $\psi:=\pi_P\cdot \varphi: G\to Q$ we have $f=k\cdot \psi$.  We claim that $\mathcal F$
is generated by $\{\psi^{-1}(Z) \mid Z\in \mathcal{N}(y)\}$, with $y:=\pi_P(x)$. Indeed, given $U\in \mathcal{N}(x)$, there exists $V \in  \mathcal{N}(e_{P})$ such that $xV \subseteq U$, and for $W \in  \mathcal{N}(e_{P}) $ with $ W^{2} \subseteq V$ we then have $\pi_P^{-1}(\pi_P(xW))=xW\overline{\{e_P\}} \subseteq x W^{2} \subseteq U$ and, consequently,
$\psi^{-1}(Z) \subseteq \varphi^{-1}(U)$ for $Z:=\pi_P(xW)$. This shows the non-trivial inclusion of the the set equality to be confirmed. \end{proof}

 \begin{proposition} \label{h-complete filter char}
The following assertions for $f:G\to H$ with $H$ Hausdorff are equivalent:
\begin{itemize}
\item[\rm (i)] $f$ is h-complete;
\item[\rm (ii)] every og-filter for $f$ has a cluster point in $G$;
\item[\rm (iii)] every maximal og-filter for $f$ converges in $G$.
\end{itemize}
\end{proposition}

\begin{proof}
(i)$\Rightarrow$(ii):   An og-filter $\mathcal F$ for $f$ comes with $r, \varphi, x$ as in Definition \ref{definition g-filter}(1), but with $\varphi: G\to P$ continuous. Since $H$ is Hausdorff, by Lemma \ref{hausdorff lemma} also $P$ may be assumed to be Hausdorff, so that also $r$ is Hausdorff and therefore $\varphi(G)$ closed, by hypothesis (i). Consequently, since $x\in \overline{\varphi(G)}$, one obtains $a\in G$ with $\varphi(a)=y$. Since $a\in \varphi^{-1}(V)$ for all $V\in \mathcal{N}(y)$, $a$ is
trivially a cluster point for $\mathcal F$.

(ii)$\Rightarrow$(iii) is obvious.

(iii)$\Rightarrow$(i): We assume $f=k\cdot h$ with $k: K\to H$ Hausdorff and consider $x \in \overline{h(G)}$. By Corollary \ref{corollary maximal}, the og-filter generated by $\{h^{-1}(U) \mid U\in \mathcal{N}(x)\}$ is contained in a maximal og-filter $\mathcal F$ for $f$ which, by hypothesis, converges to some $a\in G$. Since $h(\mathcal{F})$ converges to both $h(a)$ and $x$, and since with
$H$ and $k$ also $K$ is Hausdorff, we have $h(a)=x$ and therefore $x\in h(G)$.
\end{proof}

At the object level one again obtains the known characterization, without separation condition:

  \begin{corollary} \label{h-complete object filter char}
The following assertions for a topological group $G$ are equivalent:
\begin{itemize}
\item[\rm (i)] $G$ is h-complete;
\item[\rm (ii)] every og-filter on $G$ has a cluster point;
\item[\rm (iii)] every maximal og-filter on $G$ converges.
\end{itemize}
\end{corollary}

In analogy to Theorem \ref{c-proper filter char} we can now state the following theorem. The question whether its equivalent statements are actually weaker than those of Theorem \ref{c-proper filter char} remains open; see Question \ref{question}(3).

 \begin{theorem} \label{c-complete filter char}
The following assertions for $f:G\to H$ with $H$ Hausdorff are equivalent:
\begin{itemize}
\item[\rm (i)] for all closed $S\leq G$, $f|_S$ is h-complete;
\item[\rm (ii)] for all closed $S\leq G$, $f|_S$ is c-complete;
\item[\rm (iii)] for all closed $S\leq G$, every og-filter for $f|_S$ has a cluster point in $G$;
\item[\rm (iv)] for all closed $S\leq G$, every maximal og-filter for $f|_S$ converges in $G$.
\end{itemize}
\end{theorem}

\begin{proof}
The equivalence of (i) and (ii) follows directly from the definitions, as closed subgroups of closed subgroups are closed. The equivalence of (i), (iii), and (iv) follows from Proposition \ref{h-complete filter char} applied to $f|_S$ in lieu of $f$, with the trivial observation that cluster or convergence points in $G$ of a filter on the closed subgroup $S$ must necessarily belong to $S$.
\end{proof}

The case $H=1$ reads as follows:

\begin{corollary} \label{c-complete object filter char}
The following assertions for a topological group $G$ are equivalent:
\begin{itemize}
\item[\rm (i)] every closed subroup of $G$ is h-complete;
\item[\rm (ii)] every closed subgroup of $G$ is c-complete;
\item[\rm (iii)] every og-filter on a closed subgroup of $G$ has a cluster point in $G$;
\item[\rm (iv)] every maximal og-filter on a closed subgroup of $G$ converges in $G$.
\end{itemize}
\end{corollary}

It's time for us to make good on Remark \ref{first remark}(4) and sharpen Proposition \ref{traditional criterion}, as follows:

\begin{corollary} \label{equivalence with DU}
{\rm(1)} $f:G\to H$ with $H$ Hausdorff is c-proper if, and only if, $f\times 1_K$ is c-closed for every Hausdorff group $K$.

{\rm(2)} $G$ is c-compact if, and only if, the projection $G\times K \to K$ is c-closed for every Hausdorff group $K$.
\end{corollary}

\begin{proof}
(1) Assuming that $f\times 1_K$ is c-closed for every Hausdorff group $K$, we must show that $f$ is c-proper, and for that it suffices to verify condition (ii) of Theorem \ref{c-proper filter char}. In doing so, one may proceed as in the proof of (i)$\Rightarrow$(ii) of Theorem \ref{c-proper filter char} since, in keeping with the notation of that proof, the group $P$ may be assumed to be Hausdorff, by Lemma \ref{hausdorff lemma}.

(2) follows from (1): consider $H=1$.
\end{proof}

\begin{remark} \label{discrete c-compact}
(1) Corollary \ref{equivalence with DU} confirms the equivalence of our use of ``c-compact" with that of \cite{DU, L} for a Hausdorff groups G. For similar reasons, our use of ``h-complete" is equivalent to that of \cite{DU, L}.

(2) For $f: G\to H$ with $G$ discrete, trivially every g-filter for $f$ is actually an og-filter for $f$. By Theorem \ref{c-proper filter char} and Proposition \ref{h-complete filter char} we obtain that $f$ is c-proper if, and only if, $f|_{S}: S \to H$ is h-complete, for every subgroup $S \leq G$.
At the object level, we derive Theorem 5.3 of \cite{DU} as an immediate corollary: \em{A discrete group $G$ is c-compact  if, and only if, every subgroup $S \leq G$ is h-complete}.
\end{remark}

\begin{question}
Do Theorem {\rm{\ref{c-proper filter char}}}, Proposition {\rm{\ref{h-complete filter char}}} and Corollary {\rm{\ref{equivalence with DU}}} remain valid without the assumption that the codomain $H$ of the map $f$ be Hausdorff?
\end{question}

\section{Tychnoff type theorems for c-proper and h-complete maps}

In analogy to the Frol\'ik-Bourbaki result for (topologically) proper maps, the filter characterizations for c-proper and h-complete maps allow us to derive quite easily the stability of these properties under taking direct products, provided that the target categories are all Hausdorff.

\begin{theorem}\label{Frolik}
For a family  $(f_{i}: G_{i} \rightarrow H_{i})_{i \in I}$ of c-proper (c-perfect; h-complete) maps with every $H_i$ Hausdorff, also their product $\prod_{i\in I} f_{i}: \prod G_{i} \rightarrow \prod H_{i}$ is c-proper (c-perfect; h-complete, respectively).
\end{theorem}

\begin{proof}
To prove the stability of c-propriety, one may use (iii) of Theorem \ref{c-proper filter char} and consider a maximal g-filter
 $\mathcal F$ on a subgroup $S\leq \prod G_i$. By Proposition \ref{g-filter transfer}, its projection onto $G_i$ is a maximal g-filter on the projection of $S$ which, by hypothesis, converges in $G_i$, for every $i\in I$, thus giving a convergence point for $\mathcal F$ in $\prod G_i$, as desired. Since Hausdorffness is trivially stable under products, the same assertion for c-perfectness follows. For stability of h-completeness under products, one proceeds analogously, using Proposition \ref{h-complete filter char}.
\end{proof}

The specialization $H_i=1$ leads to the following ``Hausdorff-free" Tychonoff-style theorem that was proved in \cite{DU} for Hausdorff groups. As mentioned in Example 9.5 of \cite{CGT}, for arbitrary topological groups it may also be derived from Theorem 6.4 in \cite{CGT}, with a method that was developed in \cite{CT} for categorical closure operators.

\begin{corollary}\label{Tychnoff}
The direct product of any family of c-compact groups is c-compact. Likewise, h-completeness is preserved under taking direct products of topological groups.
\end{corollary}

This leaves us with the question whether the assumption of Hausdorffness for all $H_i \,(i\in I)$ is essential in Theorem \ref{Frolik}. Let us first point out that for finite $I$ one can prove the following statement, using the more elementary methods that were similarly applied in the proof of Prop. 4.9 in \cite{L}.

\begin{proposition}
If $f_1:G_1\to H_1, f_2:G_2\to H_2$ are both c-proper or c-perfect, then $f_1\times f_2: G_1\times G_2 \to H_1\times H_2$ has the respective property as well. Also h-completeness is preserved in this way, as long as {\rm at least one} of $H_1, H_2$ is Hausdorff.
\end{proposition}

\begin{proof}
Since $f_1\times f_2$ is the composite $(1_{H_1}\times f_2)\cdot (f_1\times 1_{G_2})$ of two c-proper maps when $f_1, f_2$ are both c-proper, the first statement is obvious. But since h-completeness is not closed under composition (see Remark \ref{second remark}(6)), it is clear that one has to argue quite differently to validate the second statement.

Let $p_i, q_i$ denote the product projections onto $G_i, H_i$ respectively,
and $s_i:G_i \to G_1\times G_2$ is the injection with $p_i\cdot s_i=1_{G_i}$, for $i=1,2$. With the hypothesis that $H_1$ be Hausdorff, we consider $h:G_1\times G_2\to K$ and $k:K\to H_1\times H_2$ with $k\cdot h=f_1\times f_2$ and $k$ Hausdorff, and without loss of generality we may assume $\overline{h(G_1\times G_2)}=K$. Then, since $s_2(G_2)$ is normal in $G_1\times G_2$, $L:=h(s_2(G_2))$ is normal in $K$. As $H_1$ is Hausdorff, so is $q_2$, and since $(q_2\cdot k)\cdot (h\cdot s_2)=f_2$, h-completeness of $f_2$ guarantees that $L$ be closed in $K$ and the quotient $K/L$ be Hausdorff.
Since $q_1(k(L))$ is trivial, we obtain $l:K/L\to H_1$ with $l\cdot \pi =q_1\cdot k$, where $\pi : K\to K/L$ is the projection. Since $l\cdot (\pi \cdot h \cdot s_1)=f_1$ with $l$ Hausdorff, h-completeness of $f_1$ makes
$\pi (h(s_1(G_1)))$ closed in $K/L$ and therefore $\pi^{-1}(\pi(h(s_1(G_1))))=h(s_1(G_1))\cdot h(s_2(G_2))= h(G_1\times G_2)$ closed in $K$.
\end{proof}

Also for infinite products of c-proper maps we can, in fact, abandon the Hausdorff restriction for the codomains altogether, by resorting to an argumentation that substantially differs from the proof of Theorem \ref{Frolik}; it extends
the finite stability via transfinite induction -- see \cite{CT}. This method was already transferred from the object to
the morphism level in Theorem 6.6 of \cite{CGT}, which requires the ambient category $\mathcal X$ to be complete and equipped
with a proper factorization system $(\mathcal E, \mathcal M)$ such that  $\mathcal E$ is a so-called projectivity class,
that is: there must be a class $\mathcal P$ of objects in $\mathcal X$ such that $\mathcal E$ is precisely the class of morphisms $f:X\to Y$ such that every object in $\mathcal P$ is projective with respect to $f$: every $y:P\to Y$
with $P\in \mathcal P$ factors as $f\cdot x=y$, for some $x:P\to X$. Furthermore, there must be a hereditary
closure operator $c=(c_X)_{X\in {\mathcal X}}$ with respect to $\mathcal M$-subobjects (in the sense of \cite {DG}) that satisfies the {\em finite structure property of products (FSPP)}, as introduced in \cite{clop}: given a product
$X= {\prod}_{i\in I} X_i$ in $\mathcal X$, the $c$-closure of an $\mathcal M$-subobject $m$ of $X$ may be computed as

$$c_X(m)=\bigwedge_{F\subseteq I \,{\rm finite}}{\pi}_F^{-1}(c_{X_F}(\pi_F(m))),$$
where $\pi_F:X\to X_F:=\prod_{i\in F}X_i$ is the generalized projection and $\pi_F(m)$ the image of $m$ according to the given factorization system.

Now, considering $\mathcal X={\bf TopGrp}$ with $(\mathcal E, \mathcal M)$ the (surjective, embedding)-factorization system and $c$ the usual Kuratowski closure operator, one notes that, of course, $\mathcal X$ is complete (as a topological category over {\bf Grp}), that $\mathcal E$ is a projectivity class (choose $\mathcal P=\{\mathbb Z\}$) and that $c$ satisfies FSPP since the the Kuratowski closure operator satisfies it in {\bf Top}.
Consequently, Theorem 6.6 of \cite{CGT} gives the following strengthening of Theorem \ref{Frolik}:

\begin{corollary}\label{Frolik categorical}
Products of c-proper morphisms of (not necessarily Hausdorff) topological groups are c-proper.
\end{corollary}

{\em Acknowledgement}. We thank G\'abor Luk\'acs for his very helpful comments on an earlier version of this paper.

\end{document}